\documentclass{amsart}

\usepackage{amssymb,color}
\usepackage[shortlabels]{enumitem}
\usepackage{bbm}
\usepackage{hyperref}
\usepackage{cleveref} 

\newtheorem{theorem}{Theorem}[section]
\newtheorem{claim}[theorem]{Claim}
\newtheorem{lemma}[theorem]{Lemma}
\newtheorem{corollary}[theorem]{Corollary}
\newtheorem{proposition}[theorem]{Proposition}
\newtheorem{question}[theorem]{Question}
\theoremstyle{definition}
\newtheorem{definition}[theorem]{Definition}

\newcommand{\Z}{\mathbb{Z}}
\newcommand{\T}{\mathbb{T}}
\newcommand{\N}{\mathbb{N}}
\newcommand{\Comp}{\mathop{\mathrm{Comp}}}
\newcommand{\R}{\mathbb{R}}
\newcommand{\M}{\mathcal{M}}
\newcommand{\res}{\upharpoonright}
\newcommand{\Char}{\mathbbm{1}}
\newcommand{\dist}{\textnormal{dist}}
\newcommand{\asdim}{\textnormal{asdim}}
\newcommand{\Succ}{\textnormal{Succ}}
\newcommand{\Int}{\textnormal{Int}}
\newcommand{\Bool}{\textnormal{Bool}}

\newcommand{\inv}{^{-1}}
\newcommand{\comp}{\mathrm{Comp}}
\newcommand{\pieces}{\mathrm{Pieces}}
\newcommand{\boldDelta}{\mathbf{\Delta}}
\newcommand{\boldSigma}{\mathbf{\Sigma}}
\newcommand{\boldPi}{\mathbf{\Pi}}
\def\mink{\overline{\dim}_{\textnormal{box}}}
\def\acts{\curvearrowright}

\usepackage{url}

\title{$G_\delta$ Circle Squaring }
\author{Spencer Unger, Narmada Varadarajan, Felix Weilacher}
\thanks{The first author is partially supported by an NSERC Discovery grant.  The third author is supported by the NSF under award number DMS-2402064.}

\begin{document}

\begin{abstract}
    We show that a disk and square of the same area in $\R^2$ are equidecomposable by translations using $\boldDelta^0_2$ pieces. That is, pieces which are simultaneously $F_\sigma$ and $G_\delta$ sets. 
    This improves a result of M\'ath\'e-Noel-Pikhurko and is the best possible complexity in terms of the Borel hierarchy. 
    More generally we show that bounded sets $A,B \subseteq \R^n$ with small enough boundaries and the same nonzero Lebesgue measure are equidecomposable with pieces that are countable unions of finite Boolean combinations of translates of $A,B$, and open sets.
    The improvement comes from constructions of low complexity toasts and related objects which should be independently useful within Borel combinatorics. 
\end{abstract}

\maketitle

\section{Introduction}

\subsection{Circle squaring}

Laczkovich \cite{laczkovich1990equidecomposability} famously showed in 1990---answering a long-standing question of Tarski\cite{Tarski}---that a (closed) disk and square in $\R^2$ of the same area are equidecomposable by translations. That is, there is an $n \in \N$, vectors $x_1,\ldots,x_n \in \R^2$, and a partition $A_1,\ldots,A_n$ of the disk, so that $A_1 + x_1,\ldots,A_n + x_n$ is a partition of the square.
Laczkovich's result has been improved upon steadily over the last decade by a series of results showing that the pieces in this equidecomposition, originally chosen in a non-constructive way, can have stronger regularity properties. 
Grabowski, M\'ath\'e, and Pikhurko \cite{GMP} showed that the pieces can be chosen to be Lebesgue measurable and to have the property of Baire. 
Marks and the first author \cite{MU} showed that the pieces can even be Borel; that is, a constructive equidecomposition of the disk and square exists.

Once we have Borel pieces, the \emph{Borel Hierarchy} becomes a natural way to measure more finely just how constructive an equidecomposition is. 
Recall that in any topological space $X$, $\boldSigma^0_1(X)$ and $\boldPi^0_1(X)$ denote the collections of open and closed subsets of $X$ respectively. 
Then, for countable ordinals $\alpha < \omega_1$ we inductively define $\boldSigma^0_{\alpha}(X)$ as the collection of countable unions of sets in $\boldPi^0_\beta(X)$ for $\beta < \alpha$ and $\boldPi^0_\alpha(X)$ as the collection of complements of sets in $\boldSigma^0_\alpha(X)$. 
Finally, $\boldDelta^0_\alpha(X)$ is the collection of sets which are in both $\boldSigma^0_\alpha(X)$ and $\boldPi^0_\alpha(X)$. 
For example, $\boldSigma^0_2(X)$ and $\boldPi^0_2(X)$ denote exactly the classes
of $F_\sigma$ and $G_\delta$ sets respectively, and $\boldDelta^0_2(X)$ denotes
the collection of sets that are simultaneously $F_\sigma$ and $G_\delta$.

If $\mathcal{B}(X)$ denotes the set of Borel subsets of $X$, then in metrizable spaces we have $\mathcal{B}(X) = \bigcup_{\alpha < \omega_1} \boldDelta^0_\alpha(X)$.
The ordinal $\alpha$ at which a given Borel set appears in this union measures how many infinitary operations are needed to produce that set starting from open sets. 
Each $\boldDelta^0_\alpha(X)$ is an algebra, so finitary operations do not raise one's complexity in this sense. 
The space $X$ in the notation above is often omitted if the ambient space is clear from context.

The pieces produced by Marks and the first author's proof are (finite) Boolean combinations of $\boldSigma^0_4$ sets; in particular they are $\boldDelta^0_5$. 
M\'ath\'e, Noel, and Pikhurko \cite{MNP} then improved this by showing that the pieces can be Boolean combinations of $\boldSigma^0_2$ sets, in particular $\boldDelta^0_3$. 
Their pieces have other impressive properties such as Jordan-measurability, but these are orthogonal to the goal of this work and we will not say more about them.
Our main result is that we can further improve the complexity of the pieces to $\boldDelta^0_2$. 

\begin{theorem}\label{thm:circle_square}
    Let $A,B \subseteq \R^2$ be a closed disk and square of the same area. $A$ and $B$ are equidecomposable by translations using $\boldDelta^0_2$ pieces. 
\end{theorem}

From the perspective of the Borel hierarchy, this is the best possible result. 
However, finer measures of complexity are possible. 
For instance, the \textit{difference hierarchy} gives a natural and well studied stratification of the class $\boldDelta^0_2$ of length $\omega_1$. Even finer is the \textit{Wadge hierarchy}, which classifies sets up to continuous reducibility (See \cite[Section 22]{kechris}). In both of these, the lowest two levels which are closed under Boolean combinations are $\boldDelta^0_1$, the class of clopen sets, and the class of Boolean combinations of open sets. Since the former is trivial in connected spaces, the following is probably the strongest question which it makes sense to ask:

\begin{question}\label{q:boolean_open}
    Let $A,B \subseteq \R^2$ be a closed disk and square of the same area. Are $A$ and $B$ equidecomposable (by translations) using pieces that are Boolean combinations of open sets?
\end{question}

By a result of Dubins, Hirsch, and Karush \cite{dubins1963scissor} a disk and square are never \textit{scissors congruent}, so the open sets involved in a positive answer to this question would still need to be somewhat complex. 

All of the previously mentioned results on so-called circle squaring are actually special cases of equidecomposability results for a much larger collection of sets in Euclidean spaces, and ours is no different. This collection is defined using the following parameter. 

\begin{definition}\label{def:box}
For any bounded set $A \subseteq \R^k$ and $\varepsilon>0$, let $N(\varepsilon,A)$ be the minimal number of boxes of side-length $\varepsilon$ needed to cover $A$.
The \emph{upper Minkowski dimension of $A$} is defined as
\begin{align*}
    \mink(A) &= \limsup_{\varepsilon \to 0}\frac{\log N(\varepsilon,S)}{\log(1/\varepsilon)}.
\end{align*}
\end{definition}

The Minkowski dimension is sometimes also called the \textit{box-counting dimension}
Clearly, if $A \subseteq \R^k$ then $\mink(A) \leq k$.
Even open sets in $\R^k$ can have boundaries with Minkowski dimension $k$, but, for instance, balls and cubes in $\R^k$ have boundaries with dimension $k-1$.

This condition turns out to be the main thing needed to generalize the circle-squaring results above: Laczkovich proved \cite{laczkovich1992small_boundary} that for any $k \in \N$, if $A,B \subseteq \R^k$ are Lebesgue measurable and bounded with $\mink(\partial A), \mink(\partial B) < k$ and $\lambda(A) = \lambda(B) > 0$ ($\lambda$ denotes the Lebesgue measure), then $A$ and $B$ are equidecomposable by translations.
The Grabowski--M\'ath\'e--Pikhurko result also applies in this case, 
as does the result of Marks and the first author if $A$ and $B$ are also Borel. 

In fact, the result of M\'ath\'e--Noel--Pikhurko produces an equidecomposition with pieces that are Boolean combinations of countable unions of Boolean combinations of translates of $A,B,$ and open sets. 
In particular, if $A,B \in \boldDelta^0_\alpha$, their pieces are Boolean combinations of $\boldSigma^0_\alpha$-sets, 
hence $\boldDelta^0_{\alpha+1}$
\footnote{It actually seems that the proof in \cite{MNP} gives $\boldDelta^0_\alpha$ pieces when $\alpha \geq 3$.}.
Our main result improves upon this general statement by one level of complexity.

\begin{theorem}\label{thm:equi} 
Let $A,B \subseteq \R^k$ be bounded sets with $\lambda(A) = \lambda(B) > 0$ and $\mink(\partial A), \mink(\partial B) <k$.  
Then,
$A$ and $B$ are equidecomposable by translations with pieces that are countable unions of Boolean combinations of translates of $A$, $B$, and open sets.
In particular, if $A,B \in \boldDelta^0_\alpha$ for $\alpha \geq 2$, so are the pieces of the equidecomposition. 
\end{theorem}

The ``in particular'' statement is shown as follows: In the main statement the pieces are $\boldSigma^0_\alpha$, but since they partition a $\boldDelta^0_\alpha$ set their complements are also $\boldSigma^0_\alpha$.

The results above are stated in terms of the action of $\R^k$ on itself by translation. However, in the proofs it is more convenient to work with the torus $\T^k := \R^k / \Z^k$. 
Here the Lebesgue measure becomes a probability measure and the notion of Minkowski dimension transfers nicely. 
It is also easy to turn an equidecomposition in $\T^k$ into one in $\R^k$. Specifically, the following are true:

\begin{lemma}\label{lem:reduce_to_torus}
    Let $A,B \subseteq [0,1/2]^k \subseteq \R^k$. Let $A',B' \subseteq \T^k$ be the images of $A$ and $B$ in the quotient $\T^k$.
    \begin{enumerate}
        \item $\mink(A) = \mink(A')$. 
        \item Let $\mathcal{A} \subseteq \mathcal{P}([0,1/2]^k)$ Suppose $A'$ and $B'$ are equidecomposable by translations (in $\T^k$) using pieces that are images of sets in $\mathcal{A}$. 
        Then $A$ and $B$ are equidecomposable by translations (in $\R^k$) using pieces that are Boolean combinations open sets and translates of sets in $\mathcal{A}$. 
    \end{enumerate}
\end{lemma}
\begin{proof}
    (1) is clear since the map $[0,1/2]^k \to \T^k$ is an isometry. 
    For (2) see e.g. \cite[Proposition 2.5]{MUnew} or \cite[Section 2]{laczkovich1993small_or_large}.
\end{proof}

Thus, for the rest of this paper we will work in the torus $\T^k$. That is, we will prove the following.

\begin{theorem}\label{thm:equi_torus}
    Let $A,B \subseteq \T^k$ with $\lambda(A) = \lambda(B) > 0$ and $\mink(\partial A), \mink(\partial B) < k$. Then, $A$ and $B$ are equidecomposable by translations using pieces that are countable unions of Boolean combinations of translates of $A$,$B$, and open sets.
\end{theorem}

Theorem \ref{thm:equi} follows from Theorem \ref{thm:equi_torus} by rescaling to assume that $\overline{A},\overline{B} \subseteq [0,1/2]^k$ (using that $A$ and $B$ are bounded) then applying Lemma \ref{lem:reduce_to_torus}. 

Finally, we remark without proof that all the results of this paper can easily be made \textit{effective}. For instance, in our equidecomposition of a disk and square in $\R^2$, we can supply for each piece $C$ a computable function $c$ from $\N^2$ to the set of open rectangles with rational coordinates such that $C = \bigcap_n \bigcup_m c(n,m)$. This is to say that Theorems \ref{thm:circle_square}, \ref{thm:equi}, and \ref{thm:equi_torus} can also be made optimal from the point of view of the so-called \textit{lightface} Borel hierarchy.

A useful though probably unnecessary fact for obtaining these effective theorems is that, by a result of Marks and the first author \cite{MUnew}, the translations $x_1,\ldots,x_d \in \T^k$ chosen randomly below can actually be chosen to be algebraic irrationals for each coordinate. Given this, verifying the claim above is simply a matter of checking that the constructions in this paper and \cite{MNP} are sufficiently uniform in the $x_i$'s, $A$, and $B$. 

\subsection{Toasts of low complexity}\label{subsec:intro_toast}

Let $x_1,\ldots,x_d \in \T^k$ for some $d$ to be specified later, 
and let $\Z^d \acts \T^k$ be the action where the $i$-th generator in $\Z^d$ acts by translation by $x_i$. Assume also that this action is free. Later the $x_i$'s will be chosen randomly, in which case this will hold with probability 1. 
Let $S = \{\gamma \in \Z^d, ||\gamma||_\infty = 1\}$, and
let $G$ denote the \textit{Schreier graph} of our action $\Z^d \acts \T^k$ with respect to $S$. That is, $V(G) = \T^k$ and $E(G) = \{(x,\gamma x) \mid x \in \T^k, \gamma \in S\}$.




This graph $G$ is the space in which all of the combinatorics of the proof from \cite{MU} take place. One of the main ingredients in their proof is the existence of a (Borel) witness to hyperfiniteness of $G$ with special properties, called a \emph{toast}. 
Toast was introduced and first constructed for Schreier graphs of Borel actions of $\Z^d$ by Gao, Jackson, Krohne, and Seward \cite{GJKS}, and it has since become a central tool in descriptive combinatorics\footnote{This may be a confusing statement given the date attached to \cite{GJKS}. Though this preprint was only recently released, its results and techniques have been circulated in the community for some time. }. 

In general a toast for a graph $G$ on $X$ is a
sequence of sets $T_0,T_1,\ldots \subseteq X$ satisfying certain geometric properties (See Definition \ref{def:toast_new}). 
In the proof from \cite{MU}, the complexity of the \textit{layers}, $T_n$, has a large bearing on the complexity of the pieces of the final equidecomposition. 
When $X$ is Polish and the action on $X$ is continuous, it is easy to check that the toast constructed in \cite{MU} has layers which are Boolean combinations of $\boldPi^0_3$ sets. In particular they are $\boldDelta^0_4$. 

In the case when $X$ is 0-dimensional, meaning it has a basis of clopen sets, the construction in \cite{MU} starts one level lower, and so produces a toast with each layer a Boolean combination of $\boldSigma^0_3$ sets.
While $\T^k$ is not 0-dimensional, M\'ath\'e--Noel--Pikhurko found a modification of this toast construction specific to translation actions on the torus which allowed them to still get a toast with $\boldSigma^0_3$-layers, which was how they saved one of the levels in their final result. 
Our further savings are entirely due to the following improvement of this using a different construction:

\begin{theorem}\label{thm:intro_torus_toast}
    Let $\Z^d \acts \T^k$ be a free action by translations with Schreier graph $G$. For any $q \in \N$, there is a $q$-toast for $G$ with each layer $\boldDelta^0_2$ 
\end{theorem}

Thus the majority of our paper is devoted to proving Theorem \ref{thm:intro_torus_toast}. Our notion of toast is slightly weaker than the one used in \cite{MNP}, so in Section \ref{sec:rounding} we will also explain how the rest of their proof still works for our toasts. 

Our toast construction follows closely the original toast construction from \cite{GJKS}, which achieved the same complexity on 0-dimensional spaces\footnote{The results in \cite{GJKS} are stated only for the free part of the Bernoulli shift $\Z^d \acts 2^{\Z^d}$, but we will see that the construction works in general.} .

\begin{theorem}[\cite{GJKS}]\label{thm:gjks_shift_toast}
    Let $\Z^d \acts X$ be a free continuous action of $\Z^d$ on a 0-dimensional Polish space with Schreier graph $G$. For any $q$, $G$ admits a $q$-toast with $\boldDelta^0_2$-layers. 
\end{theorem}

The complexity of the layers is not stated explicitly in \cite{GJKS}, but Steve Jackson indicated to us in personal communication that they should be $\boldDelta^0_2$, for which we are very grateful. In Section \ref{sec:toast-from-bgd} we will go over a part of this construction and verify this fact.


Like \cite{MNP} did with the construction in \cite{MU}, we will need to revisit the toast construction from \cite{GJKS} and introduce modifications which allow us to save a level on the torus. Our modification is based on the same topological properties of the torus as that in \cite{MNP} (see Definition \ref{def:pseudo0dim}) but the way we utilize them is quite different. Our toast construction also works in a more general situation than free actions of $\Z^d$, as is explained in the following Theorem. 

An important tool in our modification is the use of low-complexity witnesses to \textit{Borel asymptotic dimension}, a tool introduced by Conley, Jackson, Marks, Seward, and Tucker-Drob \cite{Borelasdim} which has since become central in the study of Borel group actions and descriptive combinatorics. This leads us to the following generalization of Theorem \ref{thm:gjks_shift_toast}.

\begin{theorem}\label{thm:intro_asdim_toast}
    Let $\Gamma$ be a finitely-generated group and $\Gamma \acts X$ an action with Schreier graph $G$. Suppose there exists $d \in \N$ and 
    a class of sets $\mathcal{A}$ such that $G$ admits witnesses for asymptotic dimension using sets in $\mathcal{A}$ at every scale. Then for every $q$, $G$ has a $q$-toast such that each layer and its complement are a countable union of Boolean combinations of countable unions of Boolean combinations of translates of sets in $\mathcal{A}$. 

    In particular, if the (free part of the) Bernoulli shift of $\Gamma$ admits clopen witnesses to asymptotic dimension, then any zero dimensional (free) $\Gamma$-space admits toasts with $\boldDelta^0_2$-layers, and any $\Gamma$-space admits toasts with $\boldDelta^0_3$-layers. 
\end{theorem}

Since toast has played a large role recently in many parts of descriptive combinatorics beyond just equidecomposition results, Theorem \ref{thm:intro_asdim_toast} may be of independent interest. This is helped by the fact that the hypothesis about clopen witnesses to asymptotic dimension on the free part of the Bernoulli shift has been shown to hold for a wide variety of countable amenable groups, at least in the free case \cite[Theorem 10.7]{Borelasdim}. 

There is a different construction of toast from witnesses to asymptotic dimension in \cite{Borelasdim} (although it is not labeled explicitly as such) which produces layers that are Boolean combinations of $\boldSigma^0_{\alpha+1}$ sets. 
See the discussion preceding Lemma \ref{lem:rainbow_toast} for an explanation of this. 

Theorem \ref{thm:intro_asdim_toast}
does not include Theorem \ref{thm:intro_torus_toast} as a special case since $\T^k$ is not zero dimensional. There is a corresponding generalization of Theorem \ref{thm:intro_torus_toast} in terms of witnesses to asymptotic dimension, but it requires technical and ad-hoc assumptions on the space $X$. See Theorem \ref{thm:p0d_toast}. 

Finally we mention that we do not know whether the complexity of the layers in Theorem \ref{thm:intro_asdim_toast} is optimal.

\begin{question}\label{q:toast_layers}
    Let $G$ be the Schreier graph of a free continuous action of $\Z^d$ on a 0-dimensional Polish space. 
    Does $G$
    admit a 1-toast with each layer closed? 
\end{question}

Open layers can be ruled out by an obvious compactness argument. Since most applications of toast in Borel combinatorics involve taking Boolean combinations, a positive answer to this question would not necessarily yield any complexity improvements in terms of the Borel Hierarchy. 
On the other hand, it would in some cases yield improvements as measured by the finer Hierarchies mentioned above.
However, even a toast with closed layers on the torus would not yield a positive answer to Question \ref{q:boolean_open} without improvements to other parts of the proof of circle squaring. 

\subsection{Outline of the paper}
In \Cref{sec:prelims}, we give an overview of preliminary results about Schreier graphs, complexity classes of sets, and flows. This ends with an outline of the proof of Theorem \ref{thm:equi_torus} in \Cref{subsec:toast_flows}. 
The next three sections add up to a proof of Theorems \ref{thm:intro_torus_toast} and \ref{thm:intro_asdim_toast}.
In \Cref{sec:asdim}, we cover the existence of low-complexity witnesses to asymptotic dimension.
\Cref{sec:bgd} uses these asymptotic dimension witnesses to construct an intermediate object, a \emph{bounded geometry decomposition (BGD)} (based on a similar definition from \cite{GJKS}) and in
\Cref{sec:toast-from-bgd} we follow the proof from \cite{GJKS} that a BGD can be used to construct a toast.
Finally, in \Cref{sec:rounding}, we explain the key step from \cite{MNP} used to obtain low complexity equidecompositions from a low complexity toast. 

\section{Preliminaries}\label{sec:prelims}

\subsection{The Schreier graph}

We fix the following parameters for the remainder of the paper.
Let $\Gamma$ be a finitely generated group with symmetric finite generating set $S$ and $X$ a set on which $\Gamma$ acts. 
Often we will assume this action is free. $X$ will often carry a topology $\tau$, in which case we will usually assume the action is continuous. 
In this case $X$ is called a \emph{$\Gamma$-space}.
The \textit{Schreier graph} of this action with respect to $S$ is the graph with vertex set $X$ and edges $(x, \gamma \cdot x)$ for $\gamma \in S$, $x \in X$, and will always be denoted $G$.
The connected components of the graph $G$ are the orbits of the action $\Gamma \acts X$.
We will use the following standard graph-theoretic notation:
\begin{itemize}
    \item The \emph{graph distance} $\dist(x,y)$ is the length of the shortest path in $G$ between $x$ and $y$.
    This is an extended metric on $X$.
    \item Given $R \subseteq X$, the \emph{neighbors} of $R$ are $N(R) = \{y \mid \dist(R,y) =1\}$.
    \item Given $R \subseteq X$ and $q \in \N$, the \emph{closed $q$-neighborhood of $R$} is $B(R,q) = \{y \mid \dist(x,y) \leq q \text{ for some }x \in R\}$.
    \item Given $r \in \N$, the \emph{graph power} $G^{\leq r}$ is the graph with vertex set $X$ and $(x,y) \in E(G^{\leq q})$ if $\dist(x,y) \leq r$. Equivalently, $G^{\leq r}$ is the Schreier graph with respect to the generating set $S^{\leq r}$ consisting of products of $\leq r$ elements of $S$. 
    \item Given $R \subseteq X$, the \emph{exterior boundary} of $R$ is the set $\partial_oR = (X \setminus R) \cap B(R,1)$, that is, vertices in $X \setminus R$ that are adjacent to $R$.
    Similarly, the \emph{interior boundary} is the set $\partial_i R = R \cap B(X\setminus R,1)$.
    \item The \emph{edge boundary} of $R \subseteq X$ is the set $\partial_E R = \{(x,y) \in E(G) \mid x \in R, y \not\in R \} = \{ (x,y) \in E(G)\mid x \in \partial_i(R), y \in \partial_o(R)\}$.
    \item Given $Y \subseteq X$, $G \res Y$ is the induced subgraph on the vertex set $Y$.
\end{itemize}

\subsection{Some complexity classes of sets}

\begin{definition}
    Let $\mathcal{A} \subseteq \mathcal{P}(X)$ be a collection of subsets of $X$.
    We define the following collections of sets.
    \begin{itemize}
        \item $\Bool(\mathcal{A})$ is the collection of all finite Boolean combinations of sets in $\mathcal{A}$.
        \item For $T \subseteq \Gamma$, $T \cdot \mathcal{A} = \{\gamma A \mid \gamma \in T, A \in \mathcal{A}\}$ is the set of translates of sets in $\mathcal{A}$ by elements of $T$. 
        \item $\Bool_{\Gamma}(\mathcal{A}) = \Bool(\Gamma \cdot \mathcal{A})$. 

        \item $\boldSigma_\Gamma(\mathcal{A})$ is the collection of countable unions of sets in $\Bool_\Gamma(\mathcal{A})$. 
        \item $\boldDelta_\Gamma(\mathcal{A}) = \{A \in \boldSigma_\Gamma(\mathcal{A}) \mid X \setminus A \in \boldSigma_\Gamma(\mathcal{A})\}$.
    \end{itemize}

\end{definition}

Observe that every class above with a $\Gamma$ subscript is translation-invariant, and that all of the classes above are monotone in $\mathcal{A}$. 

If $A \in \Bool_\Gamma(\mathcal{A})$ we sometimes say that $A$ is a \emph{local function} of sets in $\mathcal{A}$. 
Since our boolean combinations are finite, in this case there is some $r \in \N$ such that $A \in \Bool(S^{\leq r} \cdot \mathcal{A})$, and we say that $A$ is an \emph{$r$-local} function of sets in $\mathcal{A}$.
Equivalently, for $x \in X$ the predicate $x \in A$ only depends on the intersections of finitely many sets in $\mathcal{A}$ with $B(x,r)$.
This is often a more useful way to check if a given set $A$ is in $\Bool_{\Gamma}(\mathcal{A})$.

In what follows, we will often need to deal with not just sets of vertices, but sets of edges, as well as functions on both. 
If $F \subseteq E(G)$ is a set of (directed) edges, we will consider $F$ to be coded by the sequence of sets of vertices $\{x \mid (x, \gamma x) \in F\}$ for $\gamma \in S$. 
If $f : X \to \N$, we will consider $f$ to be coded by the sequence of fibers $f\inv(n)$ for $n \in \N$. Combining these lets us code functions on $E(G)$ similarly. 
We use these encodings to extend the notions above to these other objects. 
For example, given $f : X \to \N$, we will call any set in $\Bool_\Gamma(\{f\inv(n) \mid n \in \N\})$ a \emph{local function of $f$}. 

These definitions will be used to abstractly discuss the combinatorial properties related to the first few levels of the Borel hierarchy without need for a topology. 
For instance, the conclusion of Theorem \ref{thm:equi_torus} says that the pieces in our equidecomposition can be in $\boldDelta_\Gamma(\{A,B\} \cup \boldSigma^0_1)$. 
If $X$ is a $\Gamma$-space, every class in the Borel hierarchy is translation invariant. Thus it is clear, for example, that for any $\alpha < \omega_1$, $\boldSigma_\Gamma(\boldDelta^0_\alpha) = \boldSigma^0_{\alpha}$ and $\boldDelta_\Gamma(\boldDelta^0_\alpha) = \boldDelta^0_\alpha$. The next proposition shows that $\boldDelta_\Gamma(\mathcal{A})$ behaves well even outside of this topological setting.

\begin{proposition}\label{prop:Delta_algebra}
    For any $\mathcal{A} \subseteq \mathcal{P}(X)$, $\boldDelta_\Gamma(\mathcal{A})$ is an algebra that is invariant under the group action.
    Furthermore, $\boldDelta_\Gamma$ is idempotent: $\boldDelta_\Gamma( \boldDelta_\Gamma(\mathcal{A})) = \boldDelta_\Gamma(\mathcal{A})$.
\end{proposition}
\begin{proof}
  Invariance is clear and was noted above. 
  To show that $\boldDelta_\Gamma(\mathcal{A})$ is an algebra, we need to show it is closed under complements and finite unions.
  Note that $\boldDelta_{\Gamma}(\mathcal{A})$ is closed under complements by definition.
  So, let $A,B \in \boldDelta_{\Gamma}(\mathcal{A})$.
  Since $A, B \in \boldSigma_{\Gamma}(\mathcal{A})$, which is closed under countable unions by definition, $A \cup B \in \boldSigma_{\Gamma}(\mathcal{A})$.
  We need to show that $(A \cup B)^c \in \boldSigma_{\Gamma}(\mathcal{A})$.\footnote{When the set $X$ is clear from context, let $S^c$ denote the complement $X \setminus S$.}
  We know that $A^c, B^c \in \boldSigma_{\Gamma}(\mathcal{A})$, so we can also write $A$ and $B$ as countable \textit{intersections} of sets in $\Bool_{\Gamma}(A)$: $A = \bigcap_{i \in \N}A_i$ and $B = \bigcap_{i \in \N}B_i$.
  So,
  \begin{align*}
     (A \cup B)^c = A^c \cap B^c =  \left( \bigcup_{i \in \N} A_i^c \right) \cap \left( \bigcup_{i \in \N} B_i^c \right) = \bigcup_{i,j \in \N} \left(A_i^c \cap B_j^c\right) \in \boldSigma_{\Gamma}(\mathcal{A}),
  \end{align*}
  which shows that $A \cup B \in \boldDelta_{\Gamma}(\mathcal{A})$.

  Finally, we will show that $\boldSigma_\Gamma(\boldDelta_\Gamma(\mathcal{A})) = \boldSigma_\Gamma(\mathcal{A})$, which implies that $\boldDelta_\Gamma$ is idempotent.
  By the previous paragraph, $\Bool_\Gamma(\boldDelta_\Gamma(\mathcal{A})) = \boldDelta_\Gamma(\mathcal{A})$.
  Thus, if $X \in \boldSigma_\Gamma(\boldDelta_\Gamma(\mathcal{A}))$, then $X = \cup_{i \in \N}X_i$ for some sets $X_i \in \boldDelta_\Gamma(\mathcal{A}) \subseteq \boldSigma_\Gamma(\mathcal{A})$.
  Since $\boldSigma_\Gamma(\mathcal{A})$ is closed under countable unions, $X \in \boldSigma_\Gamma(\mathcal{A})$.
\end{proof}

We now turn to a topological notion enjoyed by the torus which will be key to our proof of Theorem \ref{thm:intro_torus_toast} and which was also an important part of the toast construction in \cite{MNP}. 

\begin{definition}\label{def:pseudo0dim}
    Let $(X, \tau)$ be a $\Gamma$-space.
    \begin{itemize}
        \item We say $A \subseteq X$ is \emph{pseudo-clopen (with respect to $\Gamma$)} if there is some $N \in \N$ so that $\partial_{\tau}A$, the topological boundary of $A$, intersects each $\Gamma$-orbit in at most $N$ points.
        \item We say $(X ,\tau)$ is \emph{pseudo-0-dimensional (with respect to $\Gamma$)} if it has a basis of pseudo-clopen sets.
    \end{itemize}
\end{definition}

We emphasize that, unlike clopenness and 0-dimensionality, pseudeo-clopenness and pseudo-0-dimensionality are properties of $X$ as a $\Gamma$-space, not as a topological space itself.
Of course, clopen sets are pseudo-clopen with respect to any group action as witnessed by $N = 0$, and so 0-dimensional spaces are always pseudo-0-dimensional.

As we have stated, in our main application $\Z^d \acts \T^k$ will be an action by independently randomly chosen translations $x_1,\ldots,x_d \in \T^k$. With probability 1, the coordinates of the $x_i$'s together with 1 will be linearly independent over $\mathbb{Q}$, which makes it clear that any axis-parallel rectangle is pseudo-clopen. Thus $\T^k$ is almost surely pseudo-0-dimensional with respect to this action. In fact, it is pseudo-0-dimensional with respect to any countable group of translations by taking rectangles with an angle independent from the translations. 


\begin{lemma}\label{lem:pc_algebra}
    The pseudo-clopen sets are an algebra that is invariant under the action of $\Gamma$. 
    That is, if $\mathcal{P}$ denotes the collection of all pseudo-clopen sets, $\Bool_{\Gamma}(\mathcal{P}) = \mathcal{P}$.
\end{lemma}
\begin{proof}
    Since $\Gamma$ acts by homeomorphisms, we have $\partial_{\tau}(\gamma A) = \gamma (\partial_{\tau}(A))$ for any $A \subseteq X$ and $\gamma \in \Gamma$, which makes invariance clear. Since $\partial_{\tau}(A^c) = \partial_\tau A$ for any $A$ we have closure under complements. Finally, for closure under finite unions, we use the fact that $\partial_{\tau}(A \cup B) \subseteq \partial_\tau A \cup \partial_\tau B$ for any $A,B$. 
\end{proof}

We can now state the promised generalization of Theorem \ref{thm:intro_torus_toast}. It will be proved in Section \ref{sec:toast-from-bgd}. Note the additional compactness assumption. 

\begin{theorem}\label{thm:p0d_toast}
    Let $\Gamma$ be a countable finitely generated group such that the Bernoulli shift of $\Gamma$ admits clopen witnesses to asymptotic dimension.
    Then the Schreier graph of any pseudo-0-dimensional compact $\Gamma$-space admits $q$-toasts with $\boldDelta^0_2$ layers for every $q \in \N$. 
\end{theorem}

\subsection{Toast, flows, and equidecompositions}\label{subsec:toast_flows}


In this section we give an overview of the proof of the main theorem on low-complexity equidecompositions from\ \cite{MNP}, and highlight the parts we need to change or improve for our main theorem. 

The first ingredient is the following definition, already discussed in Section \ref{subsec:intro_toast}, which has played a large role in the study of Borel actions of $\Z^d$ since its introduction in \cite{GJKS}. 

\begin{definition}[\cite{GJKS}]\label{def:toast_new}
    A $q$-\textit{toast} is a sequence $( T_n )_{n \in \N}$ of subsets of $X$ (called \textit{layers}) such that
    \begin{enumerate}
        \item $\bigcup_n T_n = X$. 
        \item For each $n$, the components of $G^{\leq q} \res T_n$ (called \textit{pieces}) have uniformly bounded diameter. 
        \item For each $n < m$, $\dist(T_n, \partial_i T_m) \geq q$. 
    \end{enumerate}
\end{definition}

Definitions of toast in the literature are unfortunately quite varied. Often (e.g. \cite{GR_grids}) one gives each layer $\mathcal{T}_n$ as a collection of pieces $\mathcal{T}_n \subseteq [X]^{<\infty}$, with different pieces in the same layer separated by distance $q$ and pieces $C,D$ in layers $n < m$ respectively satisfying $B(C,q) \subseteq D$ or $B(C,q) \cap D = \emptyset$. It is easy to check that these definitions are more or less equivalent upon taking $T_n = \bigcup \mathcal{T}_n$ or $\mathcal{T}_n =$ the set of components of ${G^{\leq q}} \res T_n$. Many minor variations on this version of the definition exist. For instance, in \cite[Definition 4.11]{GJKS}, there is no separation required between distinct pieces in the same layer, though they must be disjoint. 

The definition of toast in \cite[Definition 2.14]{MNP} differs from the ones discussed above in that it requires the pieces to be connected in the original graph $G$.
Our altered toast construction does not seem to produce pieces with this property, but fortunately this turns out not to be necessary for the argument to go through. 


We now turn to \textit{flows}.


\begin{definition}
    A \textit{flow} on $G$ is a function $\phi : E(G) \to \R$ satisfying $\phi(x,y) = -\phi(y,x)$ for all $(x,y) \in E(G)$. 
    A flow is called \textit{integral} if it takes values in $\Z$. 
    For a flow $\phi$ on $G$, $\phi^{out} : X \to \R$ denotes the function $x \mapsto \sum_{y \in N(x)} \phi(x,y)$. 
    For $S \subseteq X$ finite we also adopt the notation $\phi^{out}(S) = \sum_{x \in S} \phi^{out}(S) = \sum_{(x,y) \in \partial_ES} \phi(x,y)$.
    For $f : X \to \R$, $\phi$ is called an \textit{$f$-flow in $G$} if
$\phi^{out} = f$. A 0-flow is also called a \textit{circulation}. 
\end{definition}

When $G$ is the Schreier graph of an action of translations on the torus and $f
= \Char_A - \Char_B$, an $f$-flow in $G$ should be viewed as a sort of linearized
equidecomposition from $A$ to $B$.  This intuition is confirmed by Lemma
\ref{lem:flow-to-decomp} below.  The existence of an $f$-flow in $G$ is a
consequence of \emph{discrepancy estimates} for the sets $A$ and $B$ in the
action generating $G$.

\begin{definition}  For a finite $F \subseteq \mathbb{T}^k$ and Lebesgue
measurable $A \subseteq \mathbb{T}^k$, let $D(A,F) = \vert \lambda(A) -
\frac{\vert A \cap F \vert}{\vert F \vert} \vert$.  We call $D(A,F)$ the
\emph{discrepancy} of $A$ with respect to $F$. \end{definition}

\begin{definition}  Let $\mathbb{Z}^d \acts \mathbb{T}^k$ be an action by
translations $x_1, \dots, x_d$.  For $x \in \mathbb{T}^k$ and $N \in
\mathbb{N}$, let $R_N(x) = \{ x + n_1x_1 + \dots + n_dx_d \mid 0 \leq n_i < N$ for
all $i\leq n\}$. \end{definition}

Laczkovich (\cite{laczkovich1992small_boundary} Theorem 1) showed the following:

\begin{theorem} \label{thm:L-decomp} Let $x_1, \dots x_d \in \mathbb{T}^k$ be
such that the associated action $\Z^d \acts \T^k$ is free.  If $A,B \subseteq
\T^k$ have the same positive Lebesgue measure and there are constants $K$ and
$\varepsilon$ such that for all $x \in \mathbb{T}^k$ and $N>0$,
$D(A,R_N(x)),D(B,R_N(x)) \leq KN^{-1-\varepsilon}$, then $A$ and $B$ are
equidecomposible by translations.  \end{theorem}

Laczkovich also showed that for $A \subseteq \mathbb{T}^k$ with $\mink(\partial
A) < k$, for all sufficiently large $d$, almost every choice of $x_1, \dots x_d$ satisfies the hypothesis of the
theorem.  Subsequent results use
the hypothesis of Theorem \ref{thm:L-decomp} to show that the pieces of the
equidecomposition can be taken to be Lebesgue (or Baire) measurable \cite{GMP}
or Borel \cite{MU}.  The proof in \cite{MNP} uses the same
discrepancy estimates with a random choice of translations, but interestingly
the additional property of small boundary of the pieces uses a \emph{positive}
measure subset of choices of translations, see \cite{MNP} Lemma 7.5.  The paper
\cite{MUnew} provides an alternative way to produce actions with the same
discrepancy estimates.  As an application, they produce such actions with
algebraic irrational coordinates.

\begin{lemma}[{\cite{MU}, \cite[Lemma 2.16]{MNP}}]\label{lem:flow-to-decomp}
Let $A,B \subseteq \mathbb{T}^k$ and $x_1,\dots x_d \in \mathbb{T}^k$ satisfy
the hypothesis of Theorem \ref{thm:L-decomp} and let $G$ be the Schreier graph of
the associated action.  For any bounded integral $(\mathbbm{1}_A -
\mathbbm{1}_B)$-flow $\phi$ on $G$, there is an equidecomposition between $A$
and $B$ whose pieces are local functions of $\phi, A, B$, and basic open
rectangles.  \end{lemma}

Thus, for the rest of the paper we can focus on finding integral $\Char_A - \Char_B$-flows with low complexity. Our starting point is the following sequence of low complexity non-integral flows.

\begin{lemma}[{\cite[Section 4]{MU}\cite[Lemma 3.1]{MNP}}]\label{lem:approximate_flows} Let $A,B \subseteq \mathbb{T}^k$ and $x_1,\dots x_d \in \mathbb{T}^k$ satisfy
the hypothesis of Theorem \ref{thm:L-decomp} and let $G$ be the Schreier graph of
the associated action. There are $\varepsilon,c > 0$
and (bounded) flows $0 = \phi_0,\phi_1, \phi_2, \dots,$ on $G$ such that
\begin{enumerate}
    \item $2^{dm}\phi_m$ is integer valued
    \item $||\phi_m^{out} - \mathbbm{1}_A + \mathbbm{1}_B||_{\infty} \leq c2^{-m(1+\varepsilon)}$.
    \item $||\phi_m - \phi_{m+1}||_{\infty} \leq c2^{-d-\varepsilon(m-1)}$.
    \item $2^{dm}\phi_m$ is a $(2^m - 1)$ local function of $A$ and $B$.
\end{enumerate}
\end{lemma}

By (3), the pointwise limit of the $\phi_m$'s, call it $\phi_\infty$, exists and is bounded, and by (2) it is an $\Char_A - \Char_B$-flow. The Integral Flow Theorem (\ref{lem:integral-flow}) then implies that a bounded integral $\Char_A - \Char_B$-flow exists, proving Laczkovich's theorem by Lemma \ref{lem:flow-to-decomp}.

We now explain how toasts enter the picture. By (4), $\phi_\infty$ is Borel if
$A$ and $B$ are, but the integral flow theorem may not produce a Borel integral
flow. Instead, Marks and Unger show that a toast with Borel layers can be used
to ``round'' $\phi_\infty$ one layer at a time, producing successive Borel
$\Char_A - \Char_B$-flows, say $\psi_0,\psi_1,\ldots$, with $\psi_n$ integral
and equal to $\psi_{n+1}$ on the first $n$ layers. The pointwise limit of the $\psi_n$'s then exists and is a bounded integral Borel $\Char_A - \Char_B$-flow, proving the theorem.

We said above that one of the levels saved by M\'ath\'e--Noel--Pikhurko came from a construction of a simpler toast. The other comes from a refinement to the above strategy: Instead of first constructing $\phi_\infty$ and then rounding it layer by layer, they prove that one can, after passing to a subsequence, just use the simpler function $\phi_m$ on layer $m$. The following records the result:

\begin{lemma}[{\cite[Essentially Lemma 6.9]{MNP}}]\label{lem:mnp_rounding}
    Let $\Z^d \acts X$ be a free action with Schreier graph $G$ with respect to the generating set $\{\gamma \in \Z^d \mid ||\gamma||_\infty = 1\}$. 
    Let $\phi_0,\phi_1,\ldots$ be flows on $G$ satisfying (1)-(3) from Lemma \ref{lem:approximate_flows} for some $\varepsilon > 0$. 
    Let $T_0,T_1,\ldots$ be a 2-toast on $G$.
    There are a subsequence $\phi_{m_0},\phi_{m_1},\ldots$ and a bounded integral $\Char_A - \Char_B$-flow $\psi$ such that for each $n$, the restriction of $\psi$ to the set of edges meeting $T_n$ is a local-function of $T_0,\ldots,T_n$ and $\phi_{m_0},\ldots,\phi_{m_n}$. 
\end{lemma}

The indices $m_i$ just need grow fast enough as a function of $\varepsilon$ and the diameter bound for the pieces in each toast layer. See \cite[Claim 6.9.1]{MNP} and the following discussion. 
The statement of Lemma 6.9 in \cite{MNP} is for their notion of toast, which as discussed previously requires the pieces to be connected. 
This assumption is used in the proof of their Lemma 6.7 and the analogous Claim 6.9.1. 
However, it is not actually needed, as acknowledged in Footnote 1 on page 36 in \cite{MNP}. (The footnote refers to the separation of pieces within each layer, but by taking the connected components of each of our pieces, we can get a toast with connected pieces but where disitinct pieces in a layer may have distance $\leq 2$ from each other.)
We give a sketch of the proof of Lemma \ref{lem:mnp_rounding} in Section \ref{sec:rounding}. 

Putting everything in this section together tells us that simple toasts yield simple equidecompositions, as recorded by the following Corollary. This proves Theorem \ref{thm:equi_torus} upon providing the toast from Theorem \ref{thm:intro_torus_toast}

\begin{corollary}\label{cor:flow-to-equi} Let $A,B \subseteq \mathbb{T}^k$ and
$x_1,\dots x_d \in \mathbb{T}^k$ satisfy the hypothesis of Theorem
\ref{thm:L-decomp}, let $G$ be the Schreier graph of the associated action and
let $T_0,T_1,\ldots$ be a 2-toast for $G$. The sets $A$ and $B$ are
equidecomposible with pieces in $\boldDelta_\Gamma(\{A,B\} \cup \{T_n \mid n \in
\N\} \cup \boldSigma^0_1)$.  \end{corollary}

\begin{proof}
    We claim that the flow $\psi$ from Lemma \ref{lem:mnp_rounding} is coded by sets in 
    $\boldDelta_\Gamma(\{A,B\} \cup \{T_n \mid n \in \N\})$.
    This suffices by Lemma \ref{lem:flow-to-decomp} and Proposition \ref{prop:Delta_algebra}. 
    Say $\psi$ is bounded by $B \in \N$. Then $\psi$ is coded by the sets $Z_{i,\gamma} = \{x \in X \mid \psi(x,\gamma x) = i\}$ for $\gamma \in S$ and $i \in \{-B,\ldots,B\}$. For each $\gamma$, the sets $Z_{i,\gamma}$ partition $X$, 
    so it suffices to show that each is in $\boldSigma_{\Gamma}(\{A,B\} \cup \{T_n \mid n \in \N\} \cup \boldSigma^0_1)$.
    Since $\bigcup_n T_n = X$, we have 
    \[x \in Z_{i,\gamma} \Leftrightarrow \exists n\in \N, x \in T_n \wedge (\psi \res E(T_n))(x,\gamma x) = i, \]
    where $E(T_n)$ denotes the set of edges meeting $T_n$.
    For each $n$, the predicate on the right is in $\Bool_\Gamma(\{T_0,\ldots,T_n\} \cup \{A,B\})$
    by Lemma \ref{lem:mnp_rounding} and (4) from Lemma \ref{lem:approximate_flows}. 
\end{proof}

\begin{proof}[Proof of Theorem \ref{thm:equi_torus}]
    The corollary applied to a toast with $\boldDelta^0_2$ layers gives us an equidecomposition with $\boldDelta_\Gamma(\{A,B\} \cup \boldDelta^0_2)$ pieces. This class is equal to the claimed class of $\boldDelta_\Gamma(\{A,B\} \cup \boldSigma^0_1)$ by the monotonicity and idempotence of $\boldDelta_\Gamma$:
    \[ \boldDelta_{\Gamma}(\{A,B\} \cup \boldSigma^0_1) =\boldDelta_\Gamma( \boldDelta_\Gamma (\{A,B\} \cup \boldSigma^0_1)) \supseteq \boldDelta_\Gamma(\{A,B\} \cup \boldDelta^0_2) \supseteq \boldDelta_\Gamma(\{A,B\} \cup \boldSigma^0_1),\]
    where for the middle inclusion we used the fact that $\boldDelta^0_2 = \boldDelta_\Gamma(\boldSigma^0_1)$. 
\end{proof}

\section{Asymptotic Dimension}\label{sec:asdim}

Recall that $\Gamma \acts X$ is an action with Schreier graph $G$ with respect to some finite generating set $S$ for $\Gamma$.


\begin{definition}
    Let $d,r \in \N$. An \emph{asymptotic dimension $d$ witness for $G$ at scale $r$} is a a cover of $X$ by sets $U^0,\ldots,U^d$ such that for each $i$, the connected components of $G^{\leq r} \res U_i$ have uniformly bounded diameter. 

    The \emph{asymptotic dimension} of $G$ is the least $d$ such that $G$ admits such a witness at every scale $r$, or $\infty$ if there is no such $d$. We write $d = \asdim(G)$. 
    If $G$ is the Cayley graph of $\Gamma$ we write $d = \asdim(\Gamma)$. (This turns out not to depend on the choice of $S$.)
    
    For $\mathcal{A} \subseteq \mathcal{P}(X)$ we write $d = \asdim_\mathcal{A}(G)$ if $d$ is least such that an asymptotic dimension $d$ witness exists at every scale using sets from $\mathcal{A}$. 
\end{definition}

In the past few years, consideration of \textit{Borel} asymptotic dimension witnesses has played a key role in Borel combinatorics and dynamics following its introduction in \cite{Borelasdim}. 
In fact, in many situations these witnesses can be chosen to be clopen, including for actions of $\Z^d$ on 0-dimensional spaces:

\begin{theorem}[\cite{Borelasdim}, Theorem 10.7]\label{thm:clopen_asdim}
    Let $\Z^d \acts X$ be a continuous free action on a 0-dimensional second countable Hausdorff space $X$ with Schreier graph $G$. Then $\asdim_{\boldDelta^0_1}(G) = d$. 
\end{theorem}

Theorem 10.7 actually applies to a much wider class of finitely generated amenable groups $\Gamma$.

\begin{definition}\label{def:clopen_asdim}
    We say $\Gamma$ \textit{admits clopen asymptotic dimension witnesses} if there is some $d \in \N$ such that for any continuous free action $\Gamma \acts X$ on a 0-dimensional second countable Hausdorff space $X$ with Schreier graph $G$, $\asdim_{\boldDelta^0_1}(G) = d$. 
\end{definition}

By Lemma 10.3 from \cite{Borelasdim} this $d$ can always be taken to be $\asdim(\Gamma)$. For example, Theorem \ref{thm:clopen_asdim} says that $\Z^d$ admits clopen asymptotic dimension witnesses. In \cite{Borelasdim} groups of polynomial growth, polycyclic groups, and more are also shown to have this property.


The following lemma touches on the connection between descriptive combinatorics and local algorithms, and is clear to any familiar with that connection. The general phenomenon, proved separately by Bernshteyn and Seward \cite[Theorem 1.14, 1.15]{BerCts} is that for a ``local problem'', having continuous solutions is equivalent to being able to produce solutions with efficient ``local algorithms''. To avoid introducing the theory surrounding these ideas, our statement of the lemma and its proof are somewhat awkward. 

\begin{lemma}
    Let $\Gamma = \langle S \rangle$ be a countable group with $\asdim(\Gamma) = d < \infty$ and finite symmetric generating set $S$ including 1. $\Gamma$ admits clopen asymptotic dimension witnesses if and only if the following holds: For every $r \in \N$ there exists $R \in \N$ such that if $\Gamma \acts X$ is free and $c : X \to |S^R|$ is a proper coloring of $G^{\leq R}$, then there is an asymptotic dimension $d$ witness at scale $r$, say $U^0,\ldots,U^d$, with each $U^i$ a local function of $c$.
\end{lemma}

The number $|S^R|$ is there just to ensure that such a coloring always exists (by the greedy algorithm). We emphasize that the second clause in this Lemma involves no topology; it amounts to saying that the problem of finding asymptotic dimension witnesses at a given scale is an LCL, as explained in the proof below. 

\begin{proof}
    Suppose $\Gamma$ admits clopen asymptotic dimension witnesses and let $r \in \N$. The free part of the Bernoulli shift of $\Gamma$ (with base space $2^\omega)$ has a clopen partition $U^0,\ldots,U^d$ which is an asdim witness at scale $r$. Let $D \in \N$ be a bound on the diameters of the components of $G^{\leq r} \res U^i$. 

    The problem ``given a free action $\Gamma \acts X$, find a function $c : X \to (d+1)$ such that for each $i$, the diameter of each component of $G^{\leq r} \res c\inv(i)$ is at most $D$'' is a so-called LCL on $\Gamma$. This means that given a function $c : X \to (d+1)$, one can check whether $c$ satisfies the above condition by checking a fixed radius neighborhood of each point: in this case if $x \in c\inv(i)$, checking the ball $B(x,D+r)$ lets one determine whether $x$'s component in $G^{\leq r} \res c\inv(i)$ has diameter at most $D$. The conclusion now follows from the aforementioned theorem of Bernshteyn and Seward, which says that an LCL on $\Gamma$ can be solved continuously on the Bernoulli shift if and only if there is a ``local algorithm'' transforming a proper coloring into a solution just as in the conclusion of our lemma. 

    Conversely, if $\Z^d \acts X$ is a continuous free action on a 0-dimensional second countable Hausdorff space, then for any $R \in \N$ $G^{\leq R}$ admits a continuous $|S^R|$-coloring by \cite[Lemma 4.5]{bernshteyn2023distributed}, so since a local function of clopen sets is clopen, we can find clopen asymptotic dimension witnesses at each scale. 
\end{proof}

Therefore, if $\Gamma$ admits clopen asymptotic dimension witnesses, the task of finding simple witnesses to asymptotic dimension on arbitrary (i.e. not 0-dimensional) $\Gamma$-spaces reduces to that of finding simple proper colorings. The following will be used to do this in the case $X = \T^k$. 

\begin{lemma}
    Suppose $\Gamma \acts X$ is free and continuous and $X$ is a compact Polish space with basis $\mathcal{B}$. For any $R \in \N$, there is a proper coloring $c : X \to |S^R|$ of $G^{\leq R}$ with each $c\inv(i) \in \Bool_\Gamma(\mathcal{B})$. 
\end{lemma}

The proof of this Lemma uses the same strategy as that of Proposition 4.6 from \cite{KST}, which produced the same type of Borel colorings without keeping track of complexity. This is also the strategy used to prove the aforementioned fact about continuous colorings for 0-dimensional spaces. 

\begin{proof}

We start with the following claim, which does not require compactness. Let $T = S^R \setminus\{1\}$, so that $G^{\leq R}$ is the Schreier graph of our action with respect to $T$.

\begin{claim}\label{claim:discrete_basis}
    There is a basis $\mathcal{B}' \subseteq \mathcal{B}$ with each $B \in \mathcal{B'}$ $R$-discrete. 
\end{claim}

\begin{proof}
    Let $\mathcal{B}'$ be the set of $R$-discrete elements of $\mathcal{B}$. Let $x \in X$. We want to show that there is an $R$-discrete $B \in \mathcal{B}$ containing $x$.

    If not, then there are sequences $y_n,z_n$ converging to $x$ with $z_n \in T y_n$ for all $n$. By passing to a subsequence we may assume there is a fixed $\gamma \in T$ so that $z_n = \gamma y_n$ for all $n$. But then by continuity we get $\gamma x = x$, contradicting freeness. 
\end{proof}

Therefore by compactness we may find a finite cover $B_0,\ldots,B_{N-1}$ of $X$ with each $B_n$ an $R$-discrete basis set. 
We now iterate over the $B_n$, at each stage giving all points in $B_n$ the least color avaiable to them. 
More explicitly, we define an increasing sequence of proper colorings $c_i$ of $G^{\leq R} \res \bigcup_{i < n} B_i$ inductively: Given $c_n$, $c_{n+1}$ is defined on $B_n \setminus \textnormal{dom}(c_n)$ by $c_{n+1}(x) = \min (\N \setminus c_n[Tx])$. One can check inductively that each $c_n$ is proper (since $B_n$ is $R$-discrete) and all its color classes are in $\Bool_\Gamma(\{B_i \mid i < n\})$. The number of colors used is clearly at most $|T|+1 = |S^R|$. Thus $c_N$ is as desired. 

\end{proof}

\begin{corollary}\label{cor:basis_asdim}
    Suppose $\asdim(\Gamma) = d$ and $\Gamma$ admits clopen asymptotic dimension witnesses. Suppose $\Gamma \acts X$ is free and continuous and $X$ is a compact Polish space with basis $\mathcal{B}$. 
    Then $\asdim_{\Bool_\Gamma(\mathcal{B})}(G) = d$. 
\end{corollary}




\section{Bounded Geometry Decompositions}\label{sec:bgd}

Recall that $\Gamma \acts X$ is an action with Schreier graph $G$ with respect to some finite generating set $S$ for $\Gamma$.
In this section we use witnesses to asymptotic dimension to construct the objects below.

\begin{definition}\label{def:BGD}
    We say that a sequence $( X_n)_{ n \in \N}$ of subsets of $X$ is a \textit{bounded geometry decomposition} (BGD) on $G$ with constants $P$ and $Q$ if, letting $X_n^\infty = \bigcup_{k \geq n} X_k$, the following hold:
    \begin{enumerate}
        \item $\bigcap_n X_n^\infty = \emptyset$. 
        \item For each $n$, there is a uniform bound $D_n$ on the diameters of connected components of $G \res (X \setminus X_n^\infty)$. In particular these components are all finite. 
        \item For any component $R$ of $G \res (X \setminus X_n^\infty)$, the set $\{ k \geq n \mid \dist(R,X_k) \leq Q \}$ has cardinality at most $P$. 
    \end{enumerate}
\end{definition}


This definition is very close to the definition of ``Weakly orthogonal decomposition'' from \cite{GJKS}, but we have dropped their additional ``weak orthogonality'' condition, hence the name change. Weakly orthogonal decompositions are an intermediate object in Gao, Jackson, Krohne, and Seward's construction of Borel toasts for actions of $\Z^d$. In the next section we will see that BGDs can play the same role by essentially the same proof, so in this section we focus on their construction. 



Before giving our construction of BGDs, we will need to introduce yet another intermediary object. 

\begin{definition}\label{def:rainbow_toast}
    A \emph{rainbow toast of dimension $d \in \N$} on $G$ is a sequence $( U_n^i)_{n \in \N,\ 0 \leq i \leq d}$ of subsets of $X$ such that. 
    \begin{enumerate}
        \item For each $n \in \N$, $U_n^0,\ldots,U_n^d$ are an asymptotic dimension $d$ witness for $G$ at scale $n$. 
        \item For each $0 \leq i \leq d$ and $n < m \in \N$, $\dist(U_n^i, \partial_o U_m^i) \geq n$. 
    \end{enumerate}
\end{definition}

Observe that for each $0 \leq i \leq d$,
the sequence $(U_n^i)_{n \in \N}$
is something like a toast, except that not every point in $X$ is necessarily in the union. However, each point \textit{is} covered by one of these sequences.
In fact one can check that each component of $G$ is covered by a single one of these toasts. 
Rainbow toasts were constructed in \cite[Lemma 4.6]{Borelasdim} from witnesses to asymptotic dimension:

\begin{lemma}[\cite{Borelasdim}]\label{lem:rainbow_toast}
    Let $\mathcal{A} \subseteq \mathcal{P}(X)$ and $d \in \N$. Suppose $\asdim_{\mathcal{A}}(X) \leq d$. Then $X$ has a rainbow toast $( U_n^i)$ of dimension $d$ with each $U_n^i \in \Bool_{\Gamma}(\mathcal{A})$.
\end{lemma}

\begin{proof}
    One can follow proof of Lemma 4.6 from \cite{Borelasdim} while keeping track of the complexity of sets. See also the beginning of the proof of Lemma 10.5 in \cite{Borelasdim} which explicitly addresses the case where $\mathcal{A}$ is the collection of clopen sets. 
\end{proof}

The idea of the toast constructions in \cite{MNP} and \cite{MU}, which originates in \cite{BoykinJackson}, is to take a rainbow toast and choose for each component one of the toasts which covers it. The complexity of that choice is ultimately what prevents this construction from doing better than $\Bool(\boldSigma^0_2)$. Instead, we observe that a rainbow toast is quite close to a BGD. 

\begin{lemma}\label{lem:easy_BGD}
    Suppose $( U_n^i)$ is a rainbow toast on $G$ of dimension $d$. 
    For any $Q$, there is a sequence of indices $( r_n)_{ n \in \N}$ such that, letting $Y_n := \bigcup_i \partial_o U_{r_n}^i$, $( Y_n )$ is a BGD on $G$ with constants $d+2$ and $Q$.
\end{lemma}

Note that in particular $Y_n \in \Bool_\Gamma(\{U_n^i \mid i \leq d, n \in \N\})$ for each $n$. 

\begin{proof}
    We will choose our $r_n$'s inductively, starting with $r_0 = 2$. Suppose we have chosen $r_0 < \ldots < r_n$. 

    \begin{claim}\label{c:Y_n_bdd}
        $G \res X \setminus Y_n$ has connected components with a uniformly bounded diameter, call it $D_n$.
    \end{claim}

    \begin{proof}
        $D_n$ will be the bound on the diameter of the connected components of $G^{\leq r_n} \res U_{r_n}^i$ for $0 \leq i \leq d$, which exists by definition of asymptotic dimension witness.
        Let $R$ be a component of $G \res (X \setminus Y_n)$ and let $x \in R$.
        Since asymptotic dimension witnesses are covers, there is some $i$ for which $x \in U_{r_n}^i$. Let $S$ be the component of $G \res U_{r_n}^i$ containing $x$. Then $S$ has diameter at most $D_n$. We claim $R \subseteq S$. Indeed, if we follow a path in $G \res (X \setminus Y_n)$ starting at $x$, we can never leave $U_{r_n}^i$, as the first time we do we will encounter a point of $\partial_o U_{r_n}^i \subseteq Y_n$.  
    \end{proof}

    We now choose $r_{n+1} = \max(2Q + D_n, r_n) + 1$. Let us check (1)-(3) in the definition of BGD. 

    (1): We claim each $x \in X$ is in at most $d+1$ $Y_n$'s, which suffices. If not, then there are $n < m$ and $i$ for which $x \in \partial_o U_{r_n}^i \cap \partial_o U_{r_m}^i$. Then since $x$ has a neighbor in $U_{r_n}$, we get $\dist(U_{r_n}, \partial_o U_{r_m}) \leq 1$, contradicting (2) from the definition of rainbow toast.

    (2): Immediate from the Claim, since $Y_n \subseteq Y_n^\infty$. Note that the diameter bound is the same $D_n$ used to choose $r_{n+1}$.

    (3): Let $R$ be a component of $G \res (X \setminus Y_{k}^\infty)$. By (2), the diameter of $R$ is at most $D_k$. We claim that $\{n > k \mid \dist(R, Y_n) \leq Q\}$ has cardinality at most $d+1$. If not, there are $k < n < m$ and $i$ for which $\dist(R, \partial_o U_{r_n}^i), \dist(R, \partial_o U_{r_m}^i) \leq Q$. Then, as in (1), 
    \[ \dist(U_{r_n}^i, \partial_o U_{r_m}^i) \leq 1 + \textnormal{diam}(R) + 2Q \leq r_{k+1} \leq r_n, \]
    contradicting (2) from the definition of rainbow toast. 
\end{proof}

Suppose $\Gamma \acts X$ is continuous and $\asdim_{\boldDelta^0_1}(X) = d$. We saw in the previous section that this holds when $X$ is 0-dimensional and $\Gamma$ has polynomial growth. Then each $Y_n$ in the lemma above will be clopen, and so the sets $Y_n^\infty$ from the definition of BGD will be open. The constructions in the next section will be local functions of the $Y_n$ and $Y_n^\infty$, and so will end up giving a toast with $\boldDelta^0_2$ layers, matching the construction from \cite{GJKS} for $\Z^d$. 

However, the torus is not 0-dimensional, and even if our $Y_n$'s are in $\Bool(\boldSigma^0_1)$, the $Y_n^\infty$ will in general be no simpler than $\boldSigma^0_2$. This is where our notion of pseudo-0-dimensionality comes in: By taking interiors at the right moment in the proof of Lemma \ref{lem:easy_BGD}, we can actually get each $Y_n$ to be open, and thus save a level of complexity when moving from $Y_n$ to $Y_n^\infty$. 

\begin{lemma}\label{lem:pc_bgd}
    Suppose $\Gamma \acts X$ is continuous and $( U_n^i )$ is a rainbow toast of dimension $d$ on $G$ with each $U_n^i$ pseudo-clopen. For any $Q$, there is a BGD $( X_n)$ on $G$ with constants $d+2$ and $Q$, with each $X_n$ open. 
\end{lemma}

\begin{proof}
    The proof will be very similar to that of Lemma \ref{lem:easy_BGD}
    We will inductively define a sequence $r_{n}$ and let $Y_n := \bigcup_i \partial_o U_{r_n}^i$ and $X_n = \textnormal{int}(Y_n)$. We start with $r_0 = 2$. Suppose we have chosen $r_0 < \cdots < r_n$. 

    \begin{claim}
        $G \res (X \setminus X_n)$ has connected components with uniformly bounded diameter, call it $D_n$. 
    \end{claim}

    \begin{proof}
        By Claim \ref{c:Y_n_bdd}, there is a $D_n'$ bounding the diameter of components of $G \res (X \setminus Y_n)$. Note $Y_n$ is a local function of pseudo-clopen sets so it is pseudo-clopen, say as witnessed by $N \in \N$. Now fix a $\Gamma$-orbit $E$. Then $|E \cap (Y_n \setminus X_n)| \leq N$. Suppose $x,y \in E$ are in the same $G \res (X \setminus X_n)$-component, as witnessed by a $G$-path $x = x_0, \ldots, x_l = y$ avoiding $X_n$. This path also avoids $Y_n$, except for at most $N$ points, say $x_{i_j}$ where $i_1 < i_2 < \cdots < i_M$. Now for every $j$, the interval between and including $x_{i_j}$ and $x_{i_{j+1}}$ on our path has diameter at most $D_n' + 2$, and likewise for the intervals at the left and right ends. There are at most $M+1 \leq N+1$ such intervals. Thus 
        $$D_n = (N+1)(D_n' + 2)$$ is our desired bound. 
    \end{proof}

    Now we once again choose $r_{n+1} = \max(2Q + D_n,r_n) + 1$. Let us check (1) - (3) in the definition of BGD.

    (1): $\limsup Y_n = \emptyset$ as in the proof of Lemma \ref{lem:easy_BGD}, and since $X_n \subseteq Y_n$ the same is true for the $X_n$'s. 

    (2): Again, this is provided by the claim.

    (3): Let $R$ be a component of $G \res (X \setminus X_k^\infty)$. By (2), the diameter of $R$ is at most $D_k$. We claim that $S := \{n > k \mid \dist(R,X_n) \leq Q\}$ has cardinality at most $d+1$. The proof of (3) in Lemma \ref{lem:easy_BGD} shows that this bound applies to $S' := \{n > k \mid \dist(R,Y_n) \leq Q\}$, but since $X_n \subseteq Y_n$ for each $n$ we have $S \subseteq S'$ and are done. 
\end{proof}

\section{Construction of a toast from a BGD}\label{sec:toast-from-bgd}

Recall that $\Gamma \acts X$ is an action with Schreier graph $G$ with respect to some finite generating set $S$ for $\Gamma$.

The main theorem of this section is essentially Theorem 4.12 in \cite{GJKS}, except that we explicitly calculate the complexities of the sets involved.
Since our definition of a BGD drops the weak orthogonality condition from their paper and our definition of toast is a bit stronger, the proof also changes slightly but not meaningfully.
We provide the details here for completeness.

\begin{theorem}\label{thm:bgd-toast}
    Let $(X_n)_{n \in \N}$ be a BGD on $G$ with constants $P$ and $Q$, and choose $q \in \N$ so that $q(P+1) \leq Q$.
    Then, there is a $q$-toast $(T_n)_{n \in \N}$ on $G$ such that for each $n \in \N$, $T_n \in \boldDelta_\Gamma\left(\{X_n, X_n^{\infty}\mid n \in \N\} \right)$.
\end{theorem}


At the end of this section we will show how this Theorem implies our main existence results about toasts. 
Fix $q,P,Q$ as in the theorem statement, and suppose $(X_n)_{n \in \N}$ is a BGD with constants $P$ and $Q$. 

For each $n \in \N$, let $\comp(n)$ denote the set of connected components of  $G \res (X \setminus X_n^\infty)$.
Let $\comp = \bigcup_{n \in \N} \comp(n)$.
Note that for $n \leq n'$, $X \setminus X_n^\infty \subseteq X \setminus X_{n'}^\infty$, so their components must also be nested.
That is, if $R, R' \in \comp$ and $R \cap R' \neq \emptyset$, then one of them is contained in the other.
We will now construct a $q$-toast from $\comp$ by only selecting some subsets of some of the components, but first we need to get through some definitions and properties of $\comp$.

\begin{definition}
    Let $R \in \comp$. 
    The \emph{stage} of $R$ is $s(R) = \min\{n \mid R \in \comp(n)\}$. 
    The \emph{amplitude} of $R$ is $a(R) = \max\{n \mid\dist(R,X_n) \leq Q \}$, where we take $\max \emptyset = - 1$.
    $R$ is called \textit{terminal} if it is a connected component of $G$. Let $\Succ(R)$ denote the successor of $R$ in the poset $(\comp, \subseteq)$, if it exists. 
\end{definition}

\begin{lemma}\label{lem:finite_comp}
    $\Succ(R)$ exists if and only if $R$ is non-terminal.
\end{lemma}

\begin{proof}
    The elements of $\comp$ are finite by (2) in the definition of BGD (\Cref{def:BGD}). By this along with nestedness, an element of $\comp$ has a successor if and only if there is some other element of $\comp$ properly containing it.
    
    Clearly if $R$ is terminal then no other members of $\comp$ properly contain $R$. On the other hand, if $R$ is non terminal of then $\partial_o R \neq \emptyset$. By (1) in the definition of BGD, there must be some $k > s(R)$ such that $\partial_oR \not\subseteq X_k^\infty$. Then $R$ is properly contained in a member of $\comp(k)$. 
\end{proof}

\begin{lemma}\cite[Lemma~4.5]{GJKS}\label{lem:stage-amp}
    \begin{enumerate}
        \item The stage and amplitude are always finite. 
        \item If $R \subsetneq R'$ then $s(R) < s(R')$ and $a(R) \leq a(R')$.
        \item If $R$ is non terminal then $s(R) \leq a(R)$. If $R$ is terminal then $a(R) = s(R) - 1$.
        \item If $R$ is non terminal then $s(\Succ(R)) \leq a(R) + 1$. 
    \end{enumerate}
\end{lemma}
\begin{proof}
     The amplitude is finite by (3) in the definition of BGD, while for the stage there is nothing to prove. 
     If $R \subseteq R'$ have the same stage, say $n$, then they are both connected components of $G \res (X \setminus X_n^\infty)$. Then since they are not disjoint, they must be equal. The inequality $a(R) \leq a(R')$ is immediate from the definition. 

     For (3), if $R$ is non terminal then as in the previous proof we have $\emptyset \neq \partial_oR \subseteq X_{s(R)}^\infty$. Thus there is some $k \geq s(R)$ for which $X_k$ meets $\partial_o R$, and then $k \leq a(R)$ by definition of amplitude. 
     On the other hand if $R$ is terminal then $\dist(R,X_{s(R)}^\infty) = \infty$, while if $s(R) > 0$ then $X_{s(R) - 1}$ meets $R$ by definition of stage. 

    For (4), Let $k$ be as in the previous paragraph. Then $R$ is contained in a strictly larger element of $\comp(k+1)$, so $s(\Succ(R)) \leq k +1 \leq a(R) + 1$. 

\end{proof}


\begin{definition}
    For $R \in \comp$,
    We say $R$ is \emph{maximal} if it is not properly contained by any member of $\comp$ with the same amplitude. Equivalently, $R$ is terminal or $a(R) < a(\Succ(R))$.
    Let $\M$ denote the collection of all maximal components.
\end{definition}

\begin{lemma}\label{lem:max-comp}
    Every $R \in \comp$ is contained in some maximal $R' \in \M$. 
\end{lemma}
\begin{proof}
    If not, then $a(\Succ(R)) = a(\Succ^2(R)) = \dots$, while $s(R) < s(\Succ(R)) < s(\Succ^2(R)) < \dots$, so for some $n$ we will have $s(\Succ^n(R)) > a(\Succ^n(R)) + 1$, contradicting \Cref{lem:stage-amp}.
\end{proof}

It turns out that $\mathcal{M}$ already enjoys many of the properties of the collection of pieces in a toast. The last step in the proof is to show that by shrinking the elements of $\mathcal{M}$ in an appropriate way, we can obtain a toast on the nose.

\begin{definition}
    For $R \in \comp$ we let $L(R) := \{l \geq s(R) \mid \dist(R,X_l) \leq Q\}$.
\end{definition}

\begin{definition}
    Let $R \in \comp$, and let $s(R) \leq l_1 < l_2 < \dots < l_p = a(R)$ enumerate $L(R)$.
    Define
    \begin{align*}
        \Int(R) &= \{x \in R\mid \dist(x, X_{l_i}) \geq q(P-p+i), \text{ for all } 1\leq i \leq p \}.
    \end{align*}
\end{definition}

Note $L(R) = \emptyset$ if and only if $R$ is terminal, and in this case $\Int(R) = R$. 

\begin{lemma}\cite[Lemma~4.6]{GJKS}\label{lem:int}
    Let $R \in \comp$.
    \begin{enumerate}
        \item $\{x \in R\mid B(x,Q) \subseteq R\} \subseteq \Int(R)$.
        \item $\dist(\Int(R), X \setminus R) \geq q$.
        \item If $R' \in \comp$ with $R \cap R' = \emptyset$, then $\dist(\Int(R), \Int(R')) \geq 2q$.
        \item If $R' \in \comp$ with $R \subseteq R'$, then $\Int(R) \subseteq \Int(R')$.
        \item If $R \in \M$, $R' \in \comp$, and $R \subsetneq R'$, then $\dist(\Int(R), X \setminus \Int(R')) \geq q$.
    \end{enumerate}
\end{lemma}
\begin{proof}~
\begin{enumerate}[(1)]
    \item This follows from the facts that $q(P+1) \leq Q$ and $R \cap X_{l_i} = \emptyset$ for each $i=1, \dots, p$.
    \item It is equivalent to show that $\dist(\Int(R), X_{s(R)}^\infty) \geq q$.
    Suppose $x \in R$ and $\dist(x, X_{s(R)}^\infty) < q$.
    Then $\dist(x, X_{l_i}) < q$ for some $i$.
    Now $p \leq P$ by definition of $P$, 
    so $q \leq q(P - p + 1) \leq q(P - p + i)$,
    so $x \not\in\Int(R)$. 
    \item This is immediate from (2).
    \item 
    
    Let $s(R')\leq l_1' < \dots < l_{p'}' = a(R')$ enumerate $L(R')$.
    Suppose that $x \in R$ but $x \notin \Int(R')$.
    Then there is some $1 \leq i' \leq p'$ with $\dist(x,X_{l_{i'}'}) < q(P-p'+i') \leq Q$, 
    so since $x \in R$, $\dist(R, X_{l_{i'}'}) \leq Q)$.
    Since $s(R) \leq s(R')$, this implies that $l_{i'}' = l_i$ for some $1 \leq i \leq p$.
    
    We will show that $p - i \leq p' -i'$. This will imply that $\dist(x, X_{l_i}) = \dist(x, X_{l_{i'}'}) \leq q(P-p'+i') \leq q(P-p+i)$, and so $x \not\in \Int(R)$ as desired.
    For this, note that $p-i$ is the cardinality of the set 
    $L(R) \cap \{l \mid l > l_i\} = \{ l > l_i \mid \dist(R, X_l) \leq Q\}$, 
    and similarly $p'-i'$ is the cardinality of $L(R') \cap \{l \mid l > l_{i}\}= \{ l' > l_{i'}' \mid \dist(R', X_{l'}) \leq Q\}$.
    Since $R \subseteq R'$, the former set is a subset of the latter, which proves the inequality.
    
    \item The proof is similar to that of (4), and we carry over the notation. 
    Suppose $x \in \Int(R)$ and $y \in X \setminus \Int(R')$ with $\dist(x,y) < q$. By (2) we have $y \in R$ so $y \in R'$. 
    Thus there is some $i'$ such that $\dist(y,X_{l_{i'}'}) < q(P-p'+i')$.
    Then, $\dist(R, X_{l_{i'}'}) \leq \dist(x,X_{l_{i'}'}) + \dist(x,y) \leq  q(P-p'+i' +1) \leq Q$, 
    so again $l_{i'}' = l_i$ for some $i$. 

    Note that $R$ and $R'$ are non terminal since $l_i$ witnesses that $L(R),L(R') \neq \emptyset$. In particular we have $a(R') > a(R)$, so $a(R') \in L(R') \setminus L(R)$.
    
    In the analysis from (4), we still have 
    $L(R) \cap \{l \mid l > l_i\} \subseteq L(R') \cap \{l \mid l > l_i\}$
    . Moreover, the previous paragraph shows that inclusion is strict as witnessed by $a(R')$, so we get a strict inequality $p - i < p'-i'$. 
    Thus $\dist(x,X_{l_i}) < \dist(x,y) + q(P-p'+i') \leq q(P - p' + i' + 1) \leq q(P - p + i)$, 
    so $x \not\in \Int(R)$, a contradiction. 
\end{enumerate}    
\end{proof}

\begin{proof}[Proof of \Cref{thm:bgd-toast}]
 We will show that the sets
 \begin{align*}
     T_n &= \bigcup\{ \Int(R)\mid R \in \M, s(R) = n\}
 \end{align*}
 form a $q$-toast with the desired complexity.

For (1) in the definition of toast we need to show that $\bigcup_{n \in \N}T_n = X$.
For any $x \in X$, $B(x,Q)$ is finite, so by (1) in the definition of a BGD, there is some $n$ for which $B(x,Q) \cap X_n^{\infty} = \emptyset$.
Then, there is some $R \in \Comp(n)$ such that $B(x,Q) \subseteq R$, and by \Cref{lem:max-comp} there is some $R' \in \M$ such that $B(x,Q) \subseteq R \subseteq R'$.
By \Cref{lem:int}(1), $x \in \Int(R')$.
So, $x \in T_{s(R')}$.

For (2), we need to show that the components of each $G^{\leq q} \res T_n$ have uniformly bounded diameter.
Let $C$ be such a component. By \Cref{lem:int}(3), the sets $\Int(R)$ for $R \in s\inv(n) \cap \mathcal{M}$ all lie in different connected components of $G^{\leq q} \res T_n$, so $C$ is contained in one of them. Thus $D_n$ from (2) in the definition of BGD is also a bound on the diameter of $C$.

For (3), suppose $n < m$, $x \in T_n$, $y \in T_m$, and $\dist(x,y) < q$. We need to show $y \not\in \partial_i T_m$. 
That is, $N(y) \subseteq T_m$.
Let $R \in s\inv(n) \cap \M, R' \in s\inv(m) \cap M$ with $x \in \Int(R)$, $y \in \Int(R')$. By \Cref{lem:int}(3), $R$ and $R'$ are not disjoint, so by nestedness and \Cref{lem:stage-amp} we have $R \subsetneq R'$. Then by \Cref{lem:int}(5) we have $B(R,q) \subseteq \Int(R')$. Now we are done since $N(y) \subseteq B(x,q)$. 


It remains to show that $T_n \in \boldDelta\left(\{X_m, X_m^\infty\mid m \in \N\} \right)$.
We define some auxiliary sets for the proof. 
For $n \in \N$, $k \geq n-1$, we define:
\begin{align*}
    S_n &= \bigcup\{R \mid s(R) = n \}\\
    A_n^k &= \bigcup\{R \mid s(R) = n, a(R) = k\}\\
    T_n^k &= T_n \cap A_n^k = \bigcup\{\Int(R) \mid R \in \M,  s(R) =n, a(R) = k\}. 
\end{align*}
Observe that for each $n \leq k$, $T_n^k \subseteq A_n^k \subseteq S_n$.
Furthermore,
$T_n = \bigsqcup_{k \geq n}T_n^k$ and $S_n = \bigsqcup_{k \geq n}A_n^k$.
This implies that $X \setminus T_n = (X \setminus S_n) \cup \bigsqcup_{k \geq n}(A_n^k \setminus T_n^k)$.
Therefore the following claim suffices.

\begin{claim}\label{claim:sn_ank_tnk}
    Each $S_n, A_n^k$, and $T_n^k$ is in $\Bool_\Gamma(\{X_m,X_m^\infty \mid m \in \N\})$.
\end{claim}


To prove the claim,
We will show that $S_n$ is a $D_{n}$-local function of $X_n^\infty$ and $X_{n-1}$, then that $A_n^k$ and $T_n^k$ are $(D_{k+1}+Q)$-local functions of $\{X_l, X_l^\infty \mid n-1 \leq l \leq k+1\}$.
For $x \in X$, we have the following procedure to determine membership in each of these sets.
\begin{itemize}
    \item Check if $x \in X_n^\infty$. If so, then $x \notin S_n$.
    \item If not, then $x$ is in some $R \in \Comp(n)$, which can be computed from the intersection of $X_n^\infty$ and $B(x,D_n)$. 
    \item Check whether $R \notin \comp(n-1)$. Equivalently, whether $X_{n-1}$ meets $R$.
\end{itemize}
This determines whether $x \in S_n$ and
thus proves that that $S_n \in \Bool_{\Gamma}\left(\{X_n^\infty, X_{n-1}\} \right)$.
Now, suppose $x \in S_n$ and $R \in \Comp(n)$ is the unique component containing it, as above.
\begin{itemize}
    \item Check if $a(R) = k$. Equivalently, check if $X_k \cap B(R,Q) \neq \emptyset$ but $X_{k+1}^\infty \cap B(R,Q) = \emptyset$. 
\end{itemize}
This determines whether $x$ belongs to $A_n^k$, proving that $A_n^k \in \Bool_{\Gamma}\left(\{X_n^\infty, X_{n-1}, X_k, X_{k+1}^\infty\} \right)$.
Now, suppose $x \in A_n^k$.
To check if $x \in T_n^k$, we need to check if $R$ is maximal and $x \in \Int(R)$. If $k = n-1$ then $R$ is terminal so both are automatic. Otherwise $R$ is non terminal and we proceed as follows.
\begin{itemize}
    \item Compute $\Succ(R)$. By \Cref{lem:stage-amp} this has stage at most $k+1$, so we can compute it by computing the unique members of $\comp(l)$ containing $R$ for $l = n+1,\ldots,k+1$ and stopping when we first find one strictly larger than $R$. 
    \item Check whether $R$ is maximal. This is equivalent to $a(\Succ(R)) > k$, or $X_{k+1}^\infty \cap B(\Succ(R), Q) \neq \emptyset$.
    \item Compute the set $L(R) = \{l \mid n \leq l \leq k \text{ and } B(R,Q) \cap X_l \neq \emptyset \}$.
    \item For each $l \in L(R)$, compute $\dist(x,X_l)$ (noting that this is guaranteed to be at most $D_n + Q$). Use these distances to decide whether $x \in \Int(R)$. 
\end{itemize}
This proves that $T_n^k \in \Bool_{\Gamma}(\{X_l, X_l^\infty \mid n -1 \leq l \leq k+1\})$, concluding the proof of Claim \ref{claim:sn_ank_tnk} and thus of Theorem \ref{thm:bgd-toast}.
\end{proof}

\subsection{Putting it all together}

Combining the constructions of the last three sections allows us to prove our main results on the existence of toasts with simple layers.

\begin{proof}[Proof of Theorem \ref{thm:intro_asdim_toast}]
    By Lemma \ref{lem:rainbow_toast}, $G$ admits a rainbow toast of dimension $d$ using sets in $\Bool_\Gamma(\mathcal{A})$. Thus by Lemma \ref{lem:easy_BGD}, it admits a BGD $( X_n )_{n \in \N}$ with each $X_n \in \Bool_\Gamma(\mathcal{A})$
    and with constants $d+2$ and $Q$ for any $Q$. 
    Given $q \in \N$, we choose $Q = q(d+3)$.
    Then Theorem \ref{thm:bgd-toast} gives us a $q$-toast with
    each layer in 
    $\boldDelta_\Gamma(\boldSigma_\Gamma(\mathcal{A}))$. 
\end{proof}

\begin{proof}[Proof of Theorem \ref{thm:p0d_toast}, hence of Theorem \ref{thm:intro_torus_toast}]
    By Corollary \ref{cor:basis_asdim}, Lemma \ref{lem:rainbow_toast}, and Lemma \ref{lem:pc_algebra}, 
    $G$ admits a rainbow toast of dimension $d$ using pseudo-clopen sets.
    Then by Lemma \ref{lem:pc_bgd} it admits a BGD $(X_n)_{n \in \N}$ with each $X_n$ open and with constants $d+2$ and $Q$ for any $Q$. 
    Given $q \in \N$, we choose $Q = q(d+3)$.
    Then Theorem \ref{thm:bgd-toast} gives us a $q$-toast with each layer in $\boldDelta_\Gamma(\boldSigma^0_1) = \boldDelta^0_2$. 
\end{proof}

\section{The rounding algorithm}\label{sec:rounding}

In this section we sketch the proof of \Cref{lem:mnp_rounding}: the existence of a bounded integral $(\Char_A - \Char_B)$-flow $\psi$ whose restriction to the set of edges meeting the layer $T_n$ is a local function of $T_0, \dots, T_n$ and some of the approximate flows $\phi_{m_0}, \dots, \phi_{m_n}$ from \Cref{lem:approximate_flows}.
Our rounding algorithm is identical to that in \cite{MNP}.
We drop their assumption that the toast pieces are connected, but this only slightly changes the proofs of Lemma 6.7 and Claim 6.9.1. 
We provide an outline of the algorithm here, 
and explain how the proofs of these statements change.
The reader can consult  \cite[Section~6]{MNP} for more details.

The heart of the algorithm is to iteratively apply the Integral Flow Theorem to the connected components of each toast layer after ``rounding'' our flows around the boundaries of that layer's pieces. 

\begin{lemma}[The Integral Flow Theorem, {\cite[Corollary 5.2]{MU}}]\label{lem:integral-flow}
    Let $G$ be a locally finite graph on vertex set $X$, and $f:X \to \Z$ an integer valued function.
    For any $f$-flow $\phi$, there is an integral $f$-flow $\psi$ with $||\phi - \psi||_\infty \leq 1$. 
\end{lemma}



Now fix a free action $\Z^d \acts X$, $A,B \subseteq X$, a 2-toast $T_0,T_1,\ldots$ on the Schreier graph $G$, $\varepsilon > 0$, and flows $\phi_0,\phi_1,\ldots$ on on $G$ as in the statement of the Lemma.
Let $(\phi_{m_n})_{n \in \N}$ be a subsequence of approximate flows with $(m_n)_{n \in \N}$ growing very fast. 
For each $n$, let $\pieces(n)$ denote the collection of connected components of $G \res T_n$. These have bounded diameter since those of $G^{\leq 2} \res T_n$ do.

For any $F \subseteq E(G)$, let $\Delta_F$ be the graph with vertex set the undirected edges of $F$ where two elements are adjacent if they are contained in a common triangle of $G$.

\begin{lemma}[{\cite[Lemma~6.4]{MNP}}]\label{lem:circuit}
    If $S \subseteq X$, then every finite component of $\Delta_{\partial_ES}$ has an Eulerian circuit.
\end{lemma}

If $S \subseteq X$ is finite, then every component of $\Delta_{\partial_S S}$ is finite and thus has an Eulerian circuit. If $S$ is connected, then every such component has the form $\partial_E S'$ (after identifying symmetric edges) where $S'$ is either a finite component of $G \res (X \setminus S)$ (a \textit{hole} of $S$) or the union of $S$ and all these components \cite[Page 36]{MNP}

We now sketch the construction of $\psi$.
We will treat $\psi$ as mutable, updating its values as the construction proceeds.
Whenever we modify $\psi$ on some directed edge, we will always implicitly perform the opposite modification on the reverse edge, so that the result is still a flow.

\begin{enumerate}
    \item Initialize $\psi$ to be equal to $\phi_{m_n}$ on $\partial_E T_n$ for each $n$ and 0 elsewhere. 
    \item For each $n \geq 1$, we do the following in parallel for each component $S \in \pieces(n)$ and each component $C$ of $\Delta_{\partial_E S}$:
    \begin{enumerate}
    \item Find an $S'$ as in the paragraph after Lemma \ref{lem:circuit} with $C = \partial_E S'$.
    \item Find an Eulerian circuit $(u_1,v_1),\ldots,(u_t,v_t) = (u_1,v_1)$ of $C$, where for each $s$ we have $u_s \in S'$ and $v_s \not\in S'$.
    \item \emph{The adjustment step:} Add $[\psi^{out}(S')]-\psi^{out}(S')$ to $\psi(u_{t-1}, v_{t-1})$, so that the total flow leaving $S'$ is an integer.\footnote{$[x]$ denotes the nearest integer to $x \in \R$, rounding down if $x$ is a half-integer.}
    \item \emph{The rounding step:} For each $s \leq t-2$, add $[\psi(u_s,v_s)] - \psi(u_s,v_s)$ to each edge in the triangle witnessing that $(u_s,v_s)$ and $(u_{s+1},v_{s+1})$ are adjacent in $\Delta_{\partial_E S}$. 
    Since we add the same quantity along a cycle in the graph, this does not change the outflow from any vertex, i.e. we are adding a circulation at each step.

\noindent After these $t-2$ edits have been completed, the flow $\psi(u_s,v_s)$ is an integer for each $s \leq t-2$, and the total flow $\psi^{out}(S')$ is an integer, so the flow $\psi(u_{t-1}, v_{t-1})$ must also be an integer.
\end{enumerate}

    \item \emph{The completion step:} 
    The connected components of the graph $E(G) \setminus \bigcup_{n=1}^\infty \partial_E T_n$ are finite since $(T_n)$ is a toast.
    For each such component $D$, find an integral $(\Char_A - \Char_B - \psi^{out}) \res D$-flow $\psi_D$ on $G \res D$ with minimal sup norm. Add $\psi_D$ to $\psi$ on edges in $D$.
    
\end{enumerate}

In steps (b) and (3), we should make our choices using some local rule in order to get the local function part of the conclusion of Lemma \ref{lem:mnp_rounding}. 

We need to explain why a flow $\psi_D$ as in step (3) always exists. 
By the integral flow theorem, it is equivalent to show that there exists a $\Char_A - \Char_B$-flow which is equal to $\psi$ on $\partial_E D$. 
Recall that $\phi_\infty$ denotes the pointwise limit of the $\phi_m$'s and that it is a $\Char_A - \Char_B$-flow.
The rough idea for this is that for each $S'$ as in the algorithm above, if $\psi'$ denotes $\psi$ after the adjustment step, $\phi_\infty$ is so close to $\psi'$ (since $\psi$ was initialized as $\phi_{m_n}$ which was very close to $\phi_\infty$) that we can find a bounded circulation $\eta_{S'}$ (the \textit{equalizing} flow) so that $\phi_\infty + \eta_{S'} = \psi'$ on $\partial_E S'$. Therefore, if $\theta_{S'}$ denotes the circulation we add to $\psi$ in the rounding step, $\phi_\infty + \eta_{S'} + \theta_{S'}$ is a $\Char_A - \Char_B$-flow which is equal to $\psi$ on $\partial_E{S'}$. Adding these circulations for all $S'$ gives the desired witness. Again, for more details see \cite[Proof of Lemma 6.9]{MNP}. 


The following claim tells us that $\psi$ is well defined and bounded at the end of step (2) (Recall that boundedness is part of the conclusion of Lemma \ref{lem:mnp_rounding}.) It is also relevant to the conclusion that the restriction $\psi$ to the edges meeting $T_n$ is a local function of $\phi_{m_i}$ and $T_i$ for $i \leq n$. For $n \in \N$, Let us say ``step $(2)n$'' to refer to the sum of all the changes made in step (2) for $S \in \pieces(n)$. 

\begin{claim}[{\cite[Lemma 6.7]{MNP}}]
    Let $(x,y) \in E(G)$. Let $n \in \N$ be minimal with 
    $\dist(\{x,y\},T_n) \leq 1$.
    \begin{enumerate}
        \item Within step (2), $\psi(x,y)$ is altered only in step $(2)n$.
        \item The difference in $\psi(x,y)$ before and after step $(2)n$ is at most $3^d /2$. 
    \end{enumerate}
\end{claim}

\begin{proof}
    It is clear that the value $\psi(x,y)$ is changed at step $(2)m$ only if $(x,y) \in \partial_ET_m$ or there is some $z$ with $(z,x),(z,y) \in \partial_E T_m$. In particular we must have $x,y \in \partial_iT_m \cup \partial_oT_m$. This clearly fails for $m < n$ by our choice of $n$. For $m > n$ it fails by (3) in the definition of toast.  This proves part (1).

    For part (2), we claim that the rounding steps alter the value $\psi(x,y)$ at most once for each triangle in $\Delta_{\partial_E T_n}$ containing $\{x,y\}$. There are at most $3^d - 2$ such triangles, so after adding an extra $1/2$ for the adjustment step we are done. 

    Indeed, suppose first that $(x,y) \in \partial_ET_n$. Then there is a unique $S'$ as in the algorithm with $(x,y) \in \partial_E S'.$ $\psi(x,y)$ is only altered during the rounding step for $S'$, and in that step there is at most one alteration per triangle since Eulerian circuits do not repeat edges. 

    Otherwise, $\psi(x,y)$ is altered at most once for each $S'$ as above and $z \in X$ with $\{x,y,z\}$ a triangle in $G$ and $(z,x),(z,y) \in \partial_ES'$. As in the previous paragraph, the choice of $S'$ given $z$ is unique, so we are done. 
\end{proof}

This argument is where the assumption that components of $\pieces(n)$ have distance at least 3 from each other is used in \cite{MNP}. It simplifies the analysis of the final case a bit, since there is at most one $S \in \pieces(n)$ with $\{x,y\} \subseteq B(S,1)$. 

There is an analogous Claim used to show that the aforementioned equalizing flows exist and are bounded \cite[Claim 6.9.1]{MNP}. 
This is actually the spot where one needs $||\phi_{m_n} - \phi_\infty||$ to be sufficiently small compared to the diameter bound on members of $\pieces(n)$. 
The proof is essentially the same, and still works without the distance at least 3 assumption. 

This, together with the bound from the Integral Flow Theorem, shows that $\psi$ is still bounded after step (3), and by construction it is an integral $\Char_A - \Char_B$-flow, completing our sketch.

\section*{Acknowledgments}

We are very grateful to Andrew Marks and Steve Jackson for suggesting that the toast construction in \cite{GJKS} might yield simpler layers. Thanks to Forte Shinko for useful conversations regarding Section \ref{sec:asdim}.

\bibliographystyle{alphaurl}
\bibliography{bib}

@article{Borelasdim,
  title={Borel asymptotic dimension and hyperfinite equivalence relations},
  author={Conley, Clinton T and Jackson, Steve C and Marks, Andrew S and Seward, Brandon M and Tucker-Drob, Robin D},
  journal={Duke Mathematical Journal},
  volume={172},
  number={16},
  pages={3175--3226},
  year={2023},
  publisher={Duke University Press}
}

@article{MU,
author={Andrew S. Marks and Spencer T. Unger},
title={Borel circle squaring},
journal={Annals of Mathematics},
volume={186},
number={2},
pages={581--605},
year={2017},
doi={10.4007/annals.2017.186.2.4},
}

@article{BerCts,
  title={Probabilistic constructions in continuous combinatorics and a bridge to distributed algorithms},
  author={Bernshteyn, Anton},
  journal={Advances in Mathematics},
  volume={415},
  pages={108895},
  year={2023},
  publisher={Elsevier}
}

@misc{GJKS,
title={Borel combinatorics of abelian group actions},
author={Su Gao and Steve Jackson and Edward Krohne and Brandon Seward},
year={2024},
eprint={2401.13866},
archivePrefix={arXiv},
primaryClass={math.LO},
url={https://arxiv.org/abs/2401.13866},
}

@article{MNP,
  title={Circle squaring with pieces of small boundary and low borel complexity},
  author={M{\'a}th{\'e}, Andr{\'a}s and Noel, Jonathan A and Pikhurko, Oleg},
  journal={Advances in Mathematics},
  volume={484},
  pages={110685},
  year={2026},
  publisher={Elsevier}
}

@article{KST,
  title={Borel chromatic numbers},
  author={Kechris, Alexander S and Solecki, Slawomir and Todor\v{c}evi\'c, Stevo},
  journal={Advances in Mathematics},
  volume={141},
  number={1},
  pages={1--44},
  year={1999},
  publisher={Academic Press},
}

@article{laczkovich1990equidecomposability,
  title={Equidecomposability and discrepancy; a solution of Tarski's circle-squaring problem},
  author={Laczkovich, Mikl{\'o}s},
  journal={Journal f{\"u}r die reine und angewandte Mathematik},
  volume={405},
  pages={77--117},
  year={1990}
}

@article{laczkovich1992small_boundary,
  title={Decomposition of sets with small boundary},
  author={Laczkovich, Mikl{\'o}s},
  journal={Journal of the London Mathematical Society},
  volume={2},
  number={1},
  pages={58--64},
  year={1992},
  publisher={Oxford University Press}
}

@article{GMP,
  title={Measurable circle squaring},
  author={Grabowski, {\L}ukasz and M{\'a}th{\'e}, Andr{\'a}s and Pikhurko, Oleg},
  journal={Annals of Mathematics},
  pages={671--710},
  year={2017},
  publisher={JSTOR}
}

@article{Tarski,
  author    = {Alfred Tarski},
  title     = {Probléme 38},
  journal   = {Fundamenta Mathematicae},
  volume    = {7},
  pages     = {381},
  year      = {1925}
}

@article{bernshteyn2023distributed,
  title={Distributed algorithms, the {L}ov{\'a}sz {L}ocal {L}emma, and descriptive combinatorics},
  author={Bernshteyn, Anton},
  journal={Inventiones mathematicae},
  volume={233},
  number={2},
  pages={495--542},
  year={2023},
  publisher={Springer}
}

@article{BoykinJackson,
  title={Borel boundedness and the lattice rounding property},
  author={Boykin, Charles M and Jackson, Steve},
  journal={Contemporary Mathematics},
  volume={425},
  pages={113},
  year={2007},
  publisher={Providence, RI; American Mathematical Society; 1999}
}

@inproceedings{GR_grids,
  title={Classification of local problems on paths from the perspective of descriptive combinatorics},
  author={Greb{\'\i}k, Jan and Rozho{\v{n}}, V{\'a}clav},
  booktitle={Extended Abstracts EuroComb 2021: European Conference on Combinatorics, Graph Theory and Applications},
  pages={553--559},
  year={2021},
  organization={Springer}
}

@article{MUnew,
  title={A NEW PROOF OF {L}ACZKOVICH’S CIRCLE SQUARING THEOREM I},
  author={Marks, Andrew S and Unger, Spencer T},
  year={2025}
}

@article{dubins1963scissor,
  title={Scissor congruence},
  author={Dubins, Lester and Hirsch, Morris W and Karush, Jack},
  journal={Israel Journal of Mathematics},
  volume={1},
  number={4},
  pages={239--247},
  year={1963}
}

@book{kechris,
  title={Classical descriptive set theory},
  author={Kechris, Alexander},
  volume={156},
  year={2012},
  publisher={Springer Science \& Business Media}
}

@article{laczkovich1993small_or_large,
  title={Decomposition of sets of small or large boundary},
  author={Laczkovich, Mikl{\'o}s},
  journal={Mathematika},
  volume={40},
  number={2},
  pages={290--304},
  year={1993},
  publisher={London Mathematical Society}
}
\end{document}